\documentclass[11pt,a4paper]{amsart}
\usepackage{amscd,amssymb,amsopn,amsmath,amsthm,graphics,amsfonts,mathrsfs,accents,enumerate,verbatim,calc}
\usepackage[dvips]{graphicx}
\usepackage[colorlinks=true,linkcolor=red,citecolor=blue]{hyperref}
\usepackage[all]{xy}

\calclayout

\newcommand{\rt}{\rightarrow}
\newcommand{\lrt}{\longrightarrow}

\newcommand{\st}{\stackrel}
\newcommand{\pa}{\partial}
\newcommand{\al}{\alpha}

\newcommand{\la}{\lambda}
\newcommand{\Ga}{\Gamma}
\newcommand{\ga}{\gamma}

\newcommand{\lan}{\langle}
\newcommand{\ran}{\rangle}

\newcommand{\fS}{\frak{S}}

\newcommand{\C}{\mathbb{C} }
\newcommand{\D}{\mathbb{D} }
\newcommand{\K}{\mathbb{K} }

\newcommand{\R}{\mathbb{R} }
\newcommand{\Z}{\mathbb{Z} }

\newcommand{\CA}{\mathcal{A} }
\newcommand{\Ab}{\mathcal{A}\mathfrak{b} }

\newcommand{\CQ}{\mathcal{Q} }

\newcommand{\CS}{\mathcal{S} }
\newcommand{\CT}{\mathcal{T} }
\newcommand{\CX}{\mathcal{X} }
\newcommand{\CY}{\mathcal{Y} }

\newcommand{\RMod}{{\rm{Mod}} \ R}
\newcommand{\Inj}{{\rm{Inj}}}
\newcommand{\Prj}{{\rm{Prj}}}
\newcommand{\GInj}{{\rm{GInj}}}
\newcommand{\GPrj}{{\rm{GPrj}}}
\newcommand{\RGInj}{{\rm{GInj}} \ R}
\newcommand{\RGPrj}{{\rm{GPrj}} \ R}
\newcommand{\RGFlat}{{\rm{GFlat}} \ R}
\newcommand{\GFlat}{{\rm{GFlat}}}
\newcommand{\Flat}{{\rm{Flat}}}

\newcommand{\card}{{\rm{card}}}

\newcommand{\KA}{\K(\CA)}

\newcommand{\KTAP}{{\K_{\rm{tac}}({\Prj} \ \CA)}}

\newcommand{\KI}{{\K({\Inj} \ \CA)}}

\newcommand{\KTAI}{{\K_{\rm{tac}}({\Inj} \ \CA)}}

\newcommand{\KR}{\K(R)}

\newcommand{\KC}{\K(\C(R))}

\newcommand{\KPR}{{\K({\Prj} \ R)}}

\newcommand{\KTAPR}{{\K_{\rm{tac}}({\Prj} \ R)}}

\newcommand{\KIR}{{\K({\Inj} \ R)}}

\newcommand{\KTAIR}{{\K_{\rm{tac}}({\Inj} \ R)}}

\newcommand{\KFR}{{\K({\Flat} \ R)}}

\newcommand{\KPC}{{\K({\Prj} \ {\C}(R))}}
\newcommand{\KIC}{{\K({\Inj} \ {\C}(R))}}
\newcommand{\KFC}{{\K({\Flat} \ {\C}(R))}}
\newcommand{\KTAPC}{{\K_{\rm{tac}}({\Prj} \ {\C}(R))}}
\newcommand{\KTAIC}{{\K_{\rm{tac}}({\Inj} \ {\C}(R))}}

\newcommand{\KCPC}{{\K_{\rm{c}}({\Prj} \ \C(R))}}
\newcommand{\KCIC}{{\K_{\rm{c}}({\Inj} \ \C(R))}}

\newcommand{\KPQ}{{\K({\Prj} \ \CQ)}}

\newcommand{\KIQ}{{\K({\Inj} \ \CQ)}}

\newcommand{\op}{{\rm{op}}}

\newcommand{\cone}{{\rm{Cone}}}
\newcommand{\vcone}{{\rm{VCone}}}
\newcommand{\TS}{{\rm{S}}}
\newcommand{\VS}{{\rm{VS}}}

\newcommand{\id}{{\rm{id}}}

\newcommand{\HT}{{\rm{H}}}

\newcommand{\Ker}{{\rm{Ker}}}

\newcommand{\Hom}{{\rm{Hom}}}
\newcommand{\Ext}{{\rm{Ext}}}

\newtheorem{theorem}{Theorem}[section]
\newtheorem{corollary}[theorem]{Corollary}
\newtheorem{lemma}[theorem]{Lemma}
\newtheorem{proposition}[theorem]{Proposition}

\newtheorem{notation}[theorem]{Notation}

\theoremstyle{definition}
\newtheorem{definition}[theorem]{Definition}

\newtheorem{remark}[theorem]{Remark}
\newtheorem{s}[theorem]{}

\theoremstyle{plain}
\newtheorem{stheorem}{Theorem}[subsection]

\newtheorem{scorollary}[stheorem]{Corollary}
\newtheorem{slemma}[stheorem]{Lemma}
\newtheorem{sproposition}[stheorem]{Proposition}

\theoremstyle{definition}
\newtheorem{sdefinition}[stheorem]{Definition}

\newtheorem{sremark}[stheorem]{Remark}

\numberwithin{equation}{section}

\begin{document}

\title[homotopy category of projective complexes]{Homotopy category of projective complexes and complexes of Gorenstein projective modules}

\author[Asadollahi, Hafezi and Salarian]{Javad Asadollahi, Rasool Hafezi and Shokrollah Salarian}

\address{Department of Mathematics, University of Isfahan, P.O.Box: 81746-73441, Isfahan, Iran and School of Mathematics, Institute for Research in Fundamental Science (IPM), P.O.Box: 19395-5746, Tehran, Iran} \email{asadollahi@ipm.ir}
\email{r.hafezi@sci.ui.ac.ir}
\email{salarian@ipm.ir}



%\thanks{The research of the first two authors was in part supported by a grant from IPM (No. 90130216).}

\begin{abstract}
Let $R$ be a ring with identity and $\C(R)$ denote the category of complexes of $R$-modules. In this paper we study the homotopy categories arising from projective (resp. injective) complexes as well as Gorenstein projective (resp. Gorenstein injective) modules. We show that the homotopy category of projective complexes over $R$, denoted $\KPC$, is always well generated and is compactly generated provided $\KPR$ is so. Based on this result, it will be proved that the class of Gorenstein projective complexes is precovering, whenever $R$ is a commutative noetherian ring of finite Krull dimension. Furthermore, it turns out that over such rings the inclusion functor $\iota : \K(\RGPrj)\hookrightarrow \KR$ has a right adjoint $\iota_{\rho}$, where $\K(\RGPrj)$ is the homotopy category of Gorenstein projective $R$ modules. Similar, or rather dual, results for the injective (resp. Gorenstein injective) complexes will be provided. If $R$ has a dualising complex, a triangle-equivalence between homotopy categories of projective and of injective complexes will be provided. As an application, we obtain an equivalence between the triangulated categories $\K(\RGPrj)$ and $\K(\RGInj)$, that restricts to an equivalence between $\KPR$ and $\KIR$, whenever $R$ is commutative, noetherian and admits a dualising complex.
\end{abstract}

\maketitle

\tableofcontents

\section{Introduction}
Let $R$ be a ring and $\CX$ be one of the classes $\Inj R, \Prj R$ or $\Flat R$, the class of injective, projective or flat (left) $R$-modules. We call the modules in any of these classes the homological objects of $\RMod$, the category of $R$-modules.

There has been some beautiful results, starting from 2005 by papers of Krause \cite{K05} and J{\o}rgensen \cite{J05}, and continuing by papers of Iyengar and Krause \cite{IK} and Neeman \cite{N08} which focuses on the properties of homotopy category $\K(\CX)$. The upshot of these first three papers was gaining an extension of Grothendieck duality theorem. Let us be a little bite more precise. Let $R$ be commutative and noetherian admitting a special complex $D$, called the dualising complex. The Grothendieck duality theorem tells us that over such rings, the bounded derived category of finitely generated $R$-modules is self-dual, that is, there is an equivalence
\[\R\Hom_R(- , D):{\D^b(R)}^{\op} \rt \D^b(R)\]
of triangulated categories. The results of Krause and J{\o}rgensen in particular implied that for such rings, $\KIR$ and $\KPR$ are compactly generated and are infinite completions of ${\D^b(R)}^{\op}$ and $\D^b(R)$, respectively. The authors of \cite{IK} used these results in order to extend this equivalence to their completions. They showed that there is a triangle-equivalence
\[D\otimes_R - :\KPR \rt \KIR\]
which restricts on compact objects to the above equivalence.
Afterwards, the homotopy category of flat $R$-modules came into play by Neeman \cite{N08}. Among other things, he showed that $\KPR$ is always well generated and is compactly generated if $R$ is right coherent, thus obtaining a generalization of result of \cite{J05}. For a good survey on these results one may consult the introduction of \cite{N08}.

It is worth to remind that results on the compactness of the homotopy categories $\KIR$ and $\KPR$ allow us to apply Brown representability theorem for the existence of certain adjoints. These adjoints will lead to the existence of Gorenstein injective preenvelopes/Gorenstein projective precovers, see e.g. \cite{J07}.

A natural attempt is to try to get similar equivalences in another abelian categories, on one extreme and to extend the Iyengar-Krause equivalence to more larger classes, on another extreme.
Perhaps, the first attempt was done by Neeman's Ph.D. student Murfet, who generalized the above mentioned results to the category of quasi-coherent sheaves over a semi-separated noetherian scheme. He followed Neeman's beautiful idea to consider the quotient category $\KFR/\KPR^{\perp}$ as a replacement for $\KPR$ that may be extended to non-affine case, see \cite{Mu}. As another attempt, in \cite{AEHS} the authors obtained an extension of the above results in the category of representations of certain quivers. In particular, for finite quivers they presented a triangle equivalence $\KPQ \st{\sim}{\rt} \KIQ$. On the other extreme, Chen \cite[Theorem B]{C} over a left-Gorenstein ring, provides an equivalence between the triangulated categories $\K(\RGPrj)$ and $\K(\RGInj)$, that up to a natural isomorphism extends Iyengar-Krause's equivalence, when $R$ is a commutative Gorenstein ring.

The category of complexes of $R$-modules, denoted by $\C(R)$, is an abelian category with enough injective and enough projective objects. There has been several research articles dealing with the homological objects as well as their Gorenstein versions, i.e. Gorenstein projective, Gorenstein injective and Gorenstein flat complexes in this category, see e.g. \cite{R}, \cite{ER}, \cite{LZ}, \cite{EEI}.

Our aim in this paper is to study the homotopy categories of homological objects of this category and their connections to the homotopy category of Gorenstein homological objects of $R$. We may list our results in this paper as follow. For notations and terminology see the preliminaries section below.

\begin{itemize}
\item [$\bullet$] $\KPC$ is always well generated. It is compactly generated if so is $\KPR$. In this case, we present a compact generating set for $\KPC$, see Theorem \ref{main}.
\item [$\bullet$] The inclusion functor $\KPC \hookrightarrow \K(\C(R))$ admits a right adjoint, see Corollary \ref{rightadjoint}.
\item [$\bullet$] The class of Gorenstein projective complexes is precovering in $\C(R)$, provided $R$ is a commutative noetherian ring of finite Krull dimension, see Theorem \ref{Precovering}. Compare \cite{EEI}.
\item [$\bullet$] The inclusion $\iota: \K(\RGPrj) \lrt \KR$ has a right adjoint, provided $R$ is a commutative noetherian ring of finite Krull dimension, see Theorem \ref{adjoint}. Compare \cite[Theorem 2.7]{G}.
\item [$\bullet$] Similar result for injective (resp. Gorenstein injective) complexes will be discussed.
\end{itemize}
Furthermore, if $R$ is a commutative noetherian ring admitting a dualising complex, the following will be proved.
\begin{itemize}
\item [$\bullet$] There is a triangle-equivalence between triangulated categories $\KPC$ and $\KIC$, see Theorem \ref{IK-equivCompl}. This equivalence induces a triangle-equivalence between categories $\KTAPC$ and $\KTAIC$, see Proposition \ref{equivTotAcyc}.
\item [$\bullet$] The triangulated categories $\GPrj \ \underline{\C(R)}$ and $\GInj \ \overline{\C(R)}$ are compactly generated, see Corollary \ref{stableEquiv}.
\item [$\bullet$] There is a triangle-equivalence $\K(\RGPrj)\simeq \K(\RGInj)$, that restricts to an equivalence $\KPR \simeq \KIR$, see Theorem \ref{main2}.
\end{itemize}

\section{Preliminaries}
In this section we collect some of the notions and results that we need throughout the paper. Let us begin with the notion of triangulated categories.

\s {\sc Triangulated categories.}
Let $\CT$ be an additive category. $\CT$ is called a triangulate category if there exists an autoequivalence $\Sigma:\CT \rt \CT$ and a class $\Delta$ of diagrams of the form $A \rt B \rt C \rt \Sigma A$ in $\CT$ satisfying certain set of axioms. The basic references for the subject are \cite{Ve77} and \cite{N01}. We say that $\CT$ is a triangulated category with coproducts (or satisfying [TR5], in the language of Neeman) if it has arbitrary small coproducts, i.e. for any small set $\Ga$ and any collection $\{T_\ga, \ga\in \Ga\}$ of objects $T_\ga \in \CT$ indexed by $\Ga$, the categorical coproduct $\coprod_{\ga\in \Ga}T_{\ga}$ exists in $\CT$.

A triangulated subcategory $\CS$ of $\CT$ is called thick if it is closed under retracts. A thick subcategory $\CS$ of $\CT$ is called localizing (resp. colocalizing) if it is closed under all coproducts (resp. products) allowed in $\CT$. The intersection of the localizing subcategories of $\CT$ containing a class $\CS$ of objects is denoted by $\lan\CS\ran$.

\begin{s}{\sc Localization sequences.}\label{LocSeq}
Let $\CT$ be a localizing subcategory of $\CT'$. It follows from Bousfield localization (see e.g.
\cite[Theorem 9.1.13]{N01}) that the inclusion $\CT \lrt \CT'$ has a right adjoint if and only if for any object $X$ in $\CT$, there exists a triangle
$ \ \ \xymatrix@C-0.5pc{X' \ar[r] & X \ar[r] & X'' \ar@{~>}[r] & } \ $
in $\CT'$ with $X' \in \CT$ and $X'' \in \CT^\perp$. This triangle is unique up to isomorphism and the right adjoint of $\CT \lrt \CT'$ sends $X$ to $X'$. This is equivalent to say that the sequence $\CS \lrt \CT \lrt \CT/\CS$ of triangulated categories and triangulated functors is a localization sequence, \cite[\S II.2]{Ve77}. One may imagine the dual situation for colocalizing sequences. We just recall that when $\CT$ is a triangulated subcategory of $\CT'$, the left and right orthogonal of $\CT$ in $\CT'$ are defined,
respectively, by
\[ ^\perp\CT=\{ X \in \CT' \ | \ \Hom_{\CT'}(X,Y)=0, \ {\rm{for \ all}} \ Y \in \CT \},\]
\[ \CT^\perp=\{ X \in \CT' \ | \ \Hom_{\CT'}(Y,X)=0, \ {\rm{for \ all}} \ Y \in \CT \}.\]
\end{s}

\s {\sc Compactly generated triangulated categories.}\label{ComGenTriCat}
Let $\CT$ be a triangulated category with coproducts. An object $C$ of $\CT$ is called compact if any map from $C$ to an arbitrary coproduct, factors through a finite coproduct. This is equivalent to say that the representable functor $\CT(C, \ )$ commutes with coproducts. The full subcategory of all compact objects in $\CT$ form a thick subcategory that will be denoted by $\CT^c$.

Let $\CS$ be a set of objects of $\CT$. We say that $\CS$ generates $\CT$ if an object $T$ of $\CT$ is zero provided $\CT(S,T)=0$, for all $S \in \CS$. For equivalence conditions see \cite{N08}.

We say that $\CT$ is compactly generated if $\CT^c$ is essentially small and generates $\CT$.

Compactly generated triangulated categories are particularly useful. For instance, they allow us to use the Brown Representability Theorem and the Thomason Localization Theorem. A version of the Brown Representability Theorem that we shall use reads as follows. For proof see \cite[Theorem 4.1]{N96}, \cite[Theorem 8.6.1]{N01} and also
\cite[Proposition 3.3]{K05}.

\begin{lemma}\label{lBrown}
Let $F:\CT \rt \CT'$ be a triangulated functor between triangulated categories $\CT$ and $\CT'$, where $\CT$ is compactly generated.
\begin{itemize}
\item [$(1)$] $F$ admits a right adjoint if and only if it preserves all coproducts.
\item [$(2)$] $F$ admits a left adjoint if and only if it preserves all products.
\end{itemize}
\end{lemma}

\s {\sc Category of complexes.}
Let $\CX$ be an additive category. We denote by $\C(\CX)$ the category of complexes in $\CX$; the objects are complexes and morphisms are genuine chain maps. We write the complexes homologically, so an object of $\C(\CX)$ is of the following form
\[\cdots \rt X_{n+1} \st{\pa_{n+1}}{\rt} X_n  \st{\pa_{n}}{\rt} X_{n-1} \rt \cdots.\]
In case $\CA=\RMod$ is the category of (left) $R$-modules, where $R$ is an associative ring with identity, we write $\C(R)$ for $\C(\RMod)$. It is known that in case $\CA$ is additive (resp. abelian) then so is $\C(\CA)$. In particular, $\C(R)$ is an abelian category.

\s {\sc Homotopy category.}
Let $\CX$ be an additive category. The homotopy category of $\CX$, denoted $\K(\CX)$, is defined to have the same objects as in $\C(\CX)$ and morphisms are the homotopy classes of morphisms of complexes.

Let $\CA$ be an abelian category and $\CX$ be an additive subcategory of $\CA$, e.g. $\CX=\Prj\CA$, the subcategory of projective objects or $\CX=\Inj\CA$, the subcategory of injective objects of $\CA$. Then $\K(\CX)$ is a triangulated subcategory of $\CA$.

Let us consider the special case $\CA=\RMod$. In this case, we write $\KR$ for $\K(\RMod)$. If $\CX=\Prj(R)$ (resp. $\Inj(R)$, $\Flat(R)$), the subcategory of projective (resp. injective, flat) $R$-modules, then we may define the homotopy category $\KPR$, $\KIR$ and $\KFR$, respectively.

\s {\sc Total acyclicity.}\label{TotAcy}
Let $\CX$ be an additive category. A complex $X \in \C(\CX)$ is called acyclic if $\HT_n(X)=0$, for all $n \in \Z$. The triangulated subcategory of $\K(\CX)$ consisting of acyclic complexes will be denoted by $\K_{ac}(\CX)$.

A complex $X \in \C(\CX)$ is called totally acyclic if the induced complexes $\CX(X,Y)$ and $\CX(Y,X)$ of abelian groups are acyclic, for all $Y \in \CX$.

Let $\CA$ be an abelian category. If $\CX=\Prj \CA$ (resp. $\CX=\Inj \CA$), is the class of projectives (resp. injectives) in $\CA$, the objects will be called totally acyclic complexes of projectives (resp. totally acyclic complexes of injectives). The full subcategory of $\KA$ consisting of totally acyclic complexes of projectives (resp. of injectives) will be denoted by $\KTAP$ (resp. $\KTAI$).

An object $G \in \CA$ is called Gorenstein projective (resp. Gorenstein injective), if it is isomorphic to a syzygy of a totally acyclic complex of projectives (resp. injectives). We let $\GPrj \CA$ (resp. $\GInj \CA$) denote the full subcategory of $\CA$ consisting of Gorenstein projective (resp. Gorenstein injective) objects.

\s {\sc Projective, injective and flat complexes.}\label{Pro/Injcom}
Let $R$ be an associative ring with identity and $\C(R)$ denote the category of complexes over $R$. Let $X$ and $Y$ be two complexes. In what follows, $\Hom(X,Y)$ denotes the abelian group of chain maps from $X$ to $Y$. \\
A complex $P$ in $\C(R)$ is projective if the functor $\Hom(P, \ )$ is exact. This is equivalent to say that $P$ is exact and $Z_nP=\Ker(P_{n-1} \rt P_{n-2})$ is projective, for all $n \in \Z$, see \cite{R}. So, for any projective module $P$, the complex
\[\cdots \rt 0 \rt P \rt P \rt 0 \rt \cdots,\]
is projective. It is known that any projective complex can be written uniquely as a coproduct of such complexes. \\
Dually, a complex $I$ is injective if the contravariant functor $\Hom( \ ,I)$ is exact. Again it is known \cite{R} that $I$ is injective if and only if it is exact and $Z_nI$ is injective, for all $n \in \Z$. Therefore, if $I$ is an injective module, the complex
\[\cdots \rt 0 \rt I \rt I \rt 0 \rt \cdots\]
is injective. Furthermore, up to isomorphism, any injective complex is a direct product of such complexes. Note that
this direct product is in fact direct sum.

These facts imply that $\C(R)$ is an abelian category with enough projective and enough injective objects.

Finally recall that a complex $F \in \C(R)$ is flat if it is exact and for any $i \in \Z$ the module $\Ker(F_i \rt F_{i-1})$ is flat \cite[Theorem 4.1.3]{R}.

\s {\sc Gorenstein projective, Gorenstein injective and Gorenstein flat complexes.}\label{GorPro/Injcom}
Let $\C(R)$ denote the category of complexes over $R$. Based on \ref{TotAcy}, a complex $G \in \C(R)$ is called Gorenstein projective if there exists an exact sequence
\[\cdots \rt P_1 \rt P_0 \rt P_{-1} \rt \cdots \]
of projective complexes such that the sequence remains exact with respect to the functor $\Hom( \ ,P)$, for any projective complex $P$ and $G=\Ker(P_0 \rt P_{-1})$. Gorenstein injective complexes are defined dually.

We also recall that a complex $F \in \C(R)$ is called Gorenstein flat if there exists an exact sequence
\[\cdots \rt F_1 \rt F_0 \rt F_{-1} \rt F_{-2} \rt \cdots \]
such that each $F_i$ is a flat complex, $F=\Ker(F_0 \rt F_{-1})$ and the sequence remains exact under the functor $E\otimes - $, for any injective complex $E$ of right $R$-modules. For the right definition of tensor product see \cite[Theorem 4.1.3]{R}.

\s {\sc Evaluation functor and its extension.}\label{EvalFunc}
Let $\CX$ be an additive category and $\C(\CX)$ be its category of complexes. Let $i \in \Z$. There is an evaluation functor $e^i:\C(\CX) \rt \CX$, that restricts any complex $X$ to its $i$-th term $X_i$. It is known that, this functor admits both left and right adjoints, which we will denote by $e^i_{\la}$ and $e^i_{\rho}$, respectively. In fact, these adjoints are defined as follows. Given any object $M$ of $\CX$, $e^i_{\la}(M)$ (resp. $e^i_{\rho}(M)$) is defined to be the complex
\[\cdots \rt 0 \rt M \st{\id}{\rt} M \rt 0 \rt \cdots,\]
with the $M$ on the left hand side sits on the $i$-th position (resp. $(i+1)$-th term). It follows from the definition that $e^{i}_{\la}(M)=e^{i-1}_{\rho}(M)$.

When $\CX=\RMod$, we may deduce from \ref{Pro/Injcom} that, if $P$ is a projective (resp. injective) $R$-module then the complexes $e^i_{\la}(P)$ and $e^i_{\rho}(P)$ are both projective (resp. injective) complexes in $\C(R)$.

The important point here is that the functor $e^i:\C(R) \rt \RMod$ can be naturally extended to a triangulated functor
\[k^i:\K(\C(R)) \rt \KR.\]
For any complex $X$ in $\K(\C(R))$, we define $k^i(X)$ to be the complex $X^i \in \KR$ corresponds to the $i$-th row of the complex $X$. One can easily check that this extension also possesses left and right adjoints, denoted by $k^i_\la$ and $k^i_\rho$, respectively. In fact, these adjoints are the natural extensions of the functors $e^i_\la$ and $e^i_\rho$.

\s {\sc Covering/enveloping classes.}
Let $\CA$ be an abelian category and $\CX \subseteq \CA$ be a full additive subcategory which is closed under taking direct summands. Let $M$ be an object of $\CA$. A morphism $\varphi : X \rt M$ with $X \in \CX$ is called a right $\CX$-approximation (an $\CX$-precover) of $M$ if any morphism from an object $\CX$ to $M$ factors through $\varphi$. $\CX$ is called contravariantly finite (precovering) if any object in $\CA$ admits a right $\CX$-approximation. Left $\CX$-approximations ($\CX$-preenvelopes) and covariantly finite (preenveloping) subcategories are defined dually.

\s {\sc Cotorsion theory.}
Cotorsion pairs are introduced and studied by L. Salce \cite{S} in the category of abelian groups and found some more applications in different setting. They specially play an important role in the proof of the existence of flat covers. In what follows $\Ext^1(A,B)$, for $A,B \in \C(R)$, is defined to be the group of equivalence classes of all extensions
$0 \rt B \rt U \rt A \rt 0$ in $\C(R)$.

A pair $(\CX, \CY)$ of classes of objects of $\C(R)$ is said to be a cotorsion theory if $\CX^\perp=\CY$ and $\CX = {}^\perp\CY$, where the left and right orthogonal are defined as follows
\[{}^\perp\CY:=\{A \in \C(R) \ | \ \Ext^1(A,Y)=0, \ {\rm{for \ all}} \ Y \in \CY \}\]
and
\[\CX^\perp:=\{B \in \C(R) \ | \ \Ext^1(X,B)=0, \ {\rm{for \ all}} \ X \in \CX \}.\]

A cotorsion theory $(\CX, \CY)$ is called complete if for every $A \in \C(R)$ there exist exact sequences
\[0 \rt Y \rt X \rt A \rt 0 \ \ \ {\rm and} \ \ \  0 \rt A \rt Y' \rt X' \rt 0,\]
where $X, X'\in \CX$ and $Y, Y' \in \CY$.

\s \label{suspension} The following theorem of \cite[Theorem 3.5]{BEIJR} show that there is a tight connection between the complete cotorsion theories in the category of complexes of modules $\C(R)$ and the existence of adjoint functors on the corresponding homotopy categories. We will use it throughout the paper. Let us first recall a notion.

Let $X$ be a complex in $\C(R)$. The suspension of $X$, denoted by $\TS(X)$ is defined to be the shift of $X$ one degree to the left, that is for any $n \in \Z$, $\TS(X)_n=X_{n-1}$. The differentials of $\TS(X)$ are defined to be the same as the differentials of $X$ with a minus. We can define $S^i(X)$ for any $i \in \Z$. Let $\CX$ be a class of objects of $\C(R)$. We say that $\CX$ is closed under suspension, if for any $X \in \CX$ and any $i \in \Z$, we have $S^i(X) \in\CX$.\\

\noindent {\bf Theorem.}\ \cite[Theorem 3.5]{BEIJR}\label{TEnochs} Let $(\CX, \CY)$ be a complete cotorsion theory in $\C(R)$ such that $\CX$ is closed under taking suspensions. Then the inclusion functors $\K(\CX) \rt \K(\C(R))$ and $\K(\CY) \rt \K(\C(R))$ have right and left adjoints, respectively.\\

\noindent {\bf Notation.} \label{notation} Let $X$ be an object of $\K(\C(R))$. So it is a bicomplex. For any integers $i, j \in \Z$, we denote the $i$th column of $X$ by $X_i$ and the $j$th row by $X^j$. Hence $X_i^j$ denotes the module in the $i$th column and $j$th row. Indices decrease going to the right and downward. The horizontal differentials will be denoted by ${(h_x)}_i^j$ the vertical ones will be denoted by ${(v_x)}_i^j$. Note that by definition for any $i, j \in \Z$ we have ${(h_x)}_{i-1}^j {(h_x)}_i^j=0, {(v_x)}_i^{j-1} {(v_x)}_i^j=0$ and
${(h_x)}_i^{j-1}{(v_x)}_i^j={(v_x)}_{i-1}^j {(h_x)}_i^j$. We fix these notations throughout the paper.\\

\section{Homotopy category of projective complexes}
Let $R$ be ring. It was proved by J{\o}rgensen \cite{J05} that if $R$ is left and right coherent and every flat left $R$-module has finite projective dimension, then the category $\KPR$ is compactly generated. Next Neeman \cite{N08} made this result more general by reducing the assumption just to the right coherence of the ring. Moreover, he showed that $\KPR$ is always well generated. Our aim in this section is to show that for any ring $R$, the triangulated category $\KPC$ is compactly generated (resp. well generated) provided $\KPR$ is compactly generated (resp. well generated). So based on Neeman's result, $\KPC$ is always well generated and is compactly generated provided $R$ is right coherent.

Our strategy for the proof is is to build a compact generating set for $\KPC$ out of one for $\KPR$. So let $S$ be a compact generating set for $\KPR$. We show that the set  \[\fS=\{k^i_{\la}(P): P \in \CS, \ i \in \Z \}\] provides a compact generating set for $\KPC$. The compactness follows easily.

\begin{lemma}\label{kjh}
Let $P \in \KPR$ be a compact object. Then for any $i \in \Z$, $k^i_{\la}(P)$ is a compact object in $\KPC$.
\end{lemma}

\begin{proof}
We show that the functor $\Hom_{\KPC}(k^i_{\la}(P), - )$ commutes with small coproducts. Let $\{C_{\al}\}_{\al\in I}$ be a set-indexed family of objects in $\KPC$. The claim follows from the following isomorphisms. Here, for simplicity, $\Hom$ means $\Hom_{\KPC}$.
\[\begin{array}{lll}
\Hom(k^i_{\la}(P),\coprod_{\al\in I}C_{\al}) & \cong \Hom(P,(\coprod_{\al \in I}(C_{\al}))^i)\\ & \cong \coprod_{\al \in I}\Hom(P,C_{\al}^i)\\ & \cong \coprod_{\al\in I}\Hom(k^i_{\la}(P),C_{\al}).
\end{array}\]
Just note that the first and the last isomorphisms follow from adjoint property and the second one follows because $P$ is compact.
\end{proof}

For the generating set we need some lemmas. Let us begin by fixing some notations. \\

\begin{notation}\label{const}
We define two subcategories of $\KPC$. Let $\KPC^{\geq n}$ (resp. $\KPC^{\leq n}$) denote the full subcategory of $\KPC$ consisting of all complexes $P$ with the property that $P^i=0$, for all $i<n$ (resp. for all $i>n$). Both of these subcategories are closed under the formation of mapping cones and also are closed under coproducts. So they are triangulated subcategories of $\KPC$ with coproducts.\\

Let $P$ be an object of $\KPC$. By definition, it is of the form
\[P: \cdots \rt P_{i+1} \rt P_i \rt P_{i-1} \rt \cdots,\]
where $P_i$, for any $i$, is a projective complex. By \ref{EvalFunc}, for any $i \in \Z$, the complex $P_i$ can be written as a direct sum $\coprod_{j \in \Z}e^j_{\la}(P^j_i)$, where for any $i$ and $j$ in $\Z$, $P^j_i$ is a projective $R$-module.

Fix $n \in \Z$. Let $P^{\geq n}$ denote the subcomplex of $P$ that is defined as follows. The $i$th column of $P^{\geq n}$ is defined by $(P^{\geq n})_i=\coprod_{j> n}e^j_{\la}(P^j_i)$. It can be check easily that this is indeed a subcomplex. We also consider the quotient complex $P/P^{\geq n}$, whose rows sitting in grades larger than $n$ are zero. Therefore $P^{\geq n} \in \KPC^{\geq n}$ and $P/P^{\geq n} \in \KPC^{\leq n}$.

It is easy to see that the short exact sequence
\[ 0 \rt P^{\geq n} \rt P \rt P/P^{\geq n} \rt 0 \] splits in each degree and so induces a triangle
\[ \ \ \xymatrix@C-0.5pc{P^{\geq n} \ar[r] & P \ar[r] & P/P^{\geq n} \ar@{~>}[r] & } \] in $\KPC$.
\end{notation}

Recall that an abelian category $\CA$ is called cocomplete if every direct system in $\CA$ has a direct limit in $\CA$. It is known that if $\CA$ is cocomplete, then $\KA$ is a triangulated category with coproducts. $\CA$ is called complete if every inverse system in $\CA$ has an inverse limit. If $\CA$ is complete, then $\KA$ is a triangulated category with products. In fact, if $\CA$ is cocomplete (resp. complete), then so is the abelian category $\C(\CA)$ and the canonical map $\C(\CA) \rt \K(\CA)$ preserves coproducts (resp. products).

\begin{lemma} \label{Kty}
\begin{itemize}
\item [$(i)$] Let $\CA$ be a cocomplete abelian category and \[X_1 \st{\mu_1}{\rt} X_2 \st{\mu_2}{\rt} X_3 \rt \cdots \] be a sequence of morphisms in $\C(\CA)$ with each $\mu_i$ a degree-wise split monomorphism. Assume that, for any $i\geq 1$, $X_i$ is a contractible complex. Then $\underset{\underset{i\geq 1}{\longrightarrow}}{\lim} X_i$ is contractible.
\item [$(ii)$] Let $\CA$ be a complete abelian category and \[\cdots \rt X_3 \st{\mu_3}{\rt} X_2 \st{\mu_1}{\rt} X_1\] be a sequence of morphisms in $\C(\CA)$ with each $\mu_i$ a degree-wise split epimorphism. Assume that, for any $i\geq 1$, $X_i$ is a contractible complex. Then $\underset{\underset{i\geq 1}{\longleftarrow}}{\lim} X_i$ is contractible.
\end{itemize}
\end{lemma}

\begin{proof}
It follows from our assumption in part $(i)$ that we have a degree-wise split exact sequence
\[\xymatrix{0 \ar[r] & \coprod_{i\geq 1}X_i \ar[r] & \coprod_{i\geq 1}X_i \ar[r] & \displaystyle{\lim_{\lrt}}X_i \ar[r] & 0 }\]
of complexes, see \cite[Definition 1.6.4]{N01} or \cite[Remark 2.16]{M}.
This induces a triangle
\[\xymatrix{\coprod_{i\geq 1}X_i \ar[r] & \coprod_{i\geq 1}X_i \ar[r] & \displaystyle{\lim_{\lrt}}X_i \ar@{~>}[r] & }\] in $\K(\CA)$. Since for any $i \in \Z$, $X_i$ is contractible, so is the complex $\coprod_{i\geq 1}X^i$. This completes the proof of part $(i)$.

For part $(ii)$, we may use the degree-wise split exact sequence
\[\xymatrix{0 \ar[r] & \underset{\longleftarrow}{\lim} X_i \ar[r] & \prod_{i\geq 1}X_i \ar[r] & \prod_{i\geq 0}X_i \ar[r] & 0}\]
of complexes to get the induced triangle in $\K(\CA)$ and then apply the fact that $\prod_{i\geq 0}X_i$ is contractible.
\end{proof}

\begin{lemma}\label{uio}
Let $P \in \KPC^{\geq 0}$ and $Q \in \KPC^{\leq 0}$. Then
\begin{itemize}
\item [(i)] $P$ is contractible if and only if for all $i \in \Z$, $P^i$ is contractible.
\item[(ii)] $Q$ is contractible if and only if for all $i \in \Z$, $Q^i$ is contractible.
\end{itemize}
\end{lemma}

\begin{proof}
We just prove $(i)$. The proof of $(ii)$ is similar. The `only if' part is clear. For the `if' part, assume that $P^i$ is contractible, for all $i \in \Z$. In view of the above notation, we have a chain
\[\cdots \lrt P/P^{\geq 3} \st{\mu_3}{\lrt} P/P^{\geq 2} \st{\mu_1}{\lrt}  P/P^{\geq 1},\]
of quotient complexes of $P$ such that each $\mu_i$ is degree-wise split epimorphism and $P=\underset{\underset{i\geq 1}{\longleftarrow}}{\lim} P/P^{\geq i}$. The result follows by Lemma \ref{Kty}, if we show that each $P/P^{\geq i}$ is contractible. We do this by induction on $i$. Assume first that $i=1$. By definition we have $P/P^{\geq 1}=k^1_{\la}(P^1)$. Hence, the additivity of $k^1_{\la}$ in conjunction with our assumption, implies that $P/P^{\geq 1}$ is contractible. Assume inductively that $i>1$ and the result has been proved for integers smaller than $i$. The degree-wise split exact sequence
\[0 \lrt P^{\geq i-1}/P^{\geq i} \lrt P/P^{\geq i} \lrt P/P^{\geq i-1} \lrt 0\]
induces an exact triangle
\[ \ \ \xymatrix@C-0.5pc{P^{\geq i-1}/P^{\geq i} \ar[r] & P/P^{\geq i} \ar[r] & P/P^{\geq i-1} \ar@{~>}[r] & } \ \]
in $\KPC$. Hence, in view of the induction assumption, it suffices to show that $P^{\geq i-1}/P^{\geq i}$ is contractible. This follows from the fact that
\[P^{\geq i-1}/P^{\geq i} \cong k^{i}_{\la}(P^i)\] and $P^i$ is contractible. This completes the inductive step and hence the proof.
\end{proof}

Now we are ready to state and prove our main theorem in this section.

\begin{theorem}\label{main}
Let $R$ be a ring such that $\KPR$ is compactly generated. Then $\KPC$ is compactly generated. Moreover, if $\CS$ is a compact generating set for $\KPR$, the set \[\fS=\{k^i_{\la}(P): P \in \CS, \ i \in \Z \}\] provides a compact generating set for $\KPC$.
\end{theorem}

\begin{proof}
By Lemma \ref{kjh}, for any $i \in \Z$ and any $P \in \CS$, $k^i_{\la}(P)$ is compact. Set
\[\fS_1=\{k^i_{\la}(P): P \in \CS, \ i > 0 \}.\]
We claim that $\fS_1$ is a compact generating set for $\KPC^{\geq 0}$. To prove the claim, we just need to show that it is a generating set. To this end, let $X \in \KPC^{\geq 0}$ be such that $\displaystyle{\Hom_{\KPC}(k^i_{\la}(P),X)=0}$ for any element $k^i_{\la}(P)$ of $\fS_1$. By the adjoint isomorphism, we get $\Hom_{\KPR}(P,X^i)=0$ for all $i$ and all $P\in \CS$. Therefore $X^i=0$ in $\KPR$, for all $i \in \Z$. Hence Lemma \ref{uio}, implies that $X=0$. This completes the proof of the claim. Consequently $\lan\fS_1\ran=\KPC^{\geq 0}$. Similarly, one can show that the set \[\fS_2=\{k^i_{\la}(P): P \in S, \ i \leq 0 \},\] is a compact generating set for $\KPC^{\leq 0}$ and deduce that $\lan\fS_2\ran=\KPC^{\leq 0}$. Since $\fS_1 \subseteq \fS$ and $\fS_2 \subseteq \fS$, $\lan\fS_1\ran \subseteq \lan\fS\ran$ and $\lan\fS_2\ran \subseteq \lan\fS\ran$. On the other hand, in view of the Notation \ref{const}, any object $X$ of $\KPC$ fits into a triangle, in which the end terms are in $\KPC^{\geq 0}$ and $\KPC^{\leq 0}$. This implies that $\lan\fS\ran=\KPC$, or equivalently, $\fS$ is a compact generating set for $\KPC$.
\end{proof}

By \cite[Proposition 7.14]{N08}, if $R$ is right coherent, $\KPR$ is compactly generated. So in view of the above theorem, we have the following result.

\begin{corollary}\label{RightCoherent}
Let $R$ be a right coherent ring. Then $\KPC$ is compactly generated.
\end{corollary}

\s {\sc Well generated triangulated categories.}
Well generated triangulated categories are introduced by Neeman \cite{N01}, in order to provide a vast generalization of Brown's representability theorem.

Let $\CT$ be a triangulated category with coproducts. An object $C$ of $\CT$ is called $\al$-small, for some regular cardinal $\al$, if every map $C \rt \coprod_{i \in I}T_i$ factors through $\coprod_{i \in J}T_i$, for some $J \subseteq I$ with $\card J<\al$. The full subcategory $\CT$ consisting of all $\al$-small objects is denoted by $\CT^{\al}$.

Let $\CS$ be a triangulated subcategory of $\CT$. $\CS$ is called $\al$-localizing if any coproduct of fewer than $\al$ objects of $\CS$ lies in $\CS$. It is clear that $\CS$ is localizing if it is $\al$-localizing for every infinite cardinal $\al$.

For any cardinal $\al$ we use ${\lan\CS\ran}^{\al}$ to denote the smallest $\al$-localizing subcategory of $\CT$ containing $\CS$. Again it is clear that \[\lan\CS\ran=\bigcup_{\al}{\lan\CS\ran}^{\al}.\]

\begin{definition}
Let $\CT$ be a triangulated category with coproducts and $\al$ be a regular cardinal. $\CT$ is called $\al$-compactly generated if the category $\CT^{\al}$ is essentially small and ${\lan\CT^{\al}\ran}^{\al}=\CT$. $\CT$ is called well generated if it is $\al$-compactly generated, for some regular cardinal $\al$.
\end{definition}

It is clear that any well generated triangulated category is $\beta$-compactly generated, for all sufficiently large $\beta$. The well generated categories share many important properties with the compactly generated triangulated categories. For an equivalent definition see \cite[Theorem A]{K01}.\\

\begin{corollary}
Let $R$ be a ring. Then $\KPC$ is well generated.
\end{corollary}

\begin{proof}
It was proved by Neeman \cite[Theorem 5.9]{N08} that $\KPR$ is always well generated. Now one may follow the above argument to deduce the result.
\end{proof}

This, in particular implies the following corollary.

\begin{corollary}\label{rightadjoint}
For any ring $R$, the inclusion $\iota: \KPC \rt \K(\C(R))$ has a right adjoint.
\end{corollary}

\begin{proof}
The result follows from \cite[Corollary 5.10]{N08}, because $\KPC$ is well generated and $\iota$ preserves coproducts. \end{proof}

\begin{remark}
The results of this section can be stated in a dual manner to prove that $\KIC$ is compactly generated, provided $\KIR$ is compactly generated. In fact, if $\CS$ is a compact generating set for $\KIR$, the set
\[\fS=\{k^i_{\rho}(I): I \in \CS, \ i \in \Z \}\] provides a compact generating set for $\KIC$.

We know by \cite{K05} that $\KIR$ is compactly generated, provided $R$ is a noetherian ring. So in this case, $\KIC$ is compactly generated. Note that if $R$ is a noetherian ring, $\C(R)$ is a locally noetherian Grothendieck category. So one may conclude directly from Proposition 2.3 of \cite{K05} that $\KIC$ is compactly generated. However, here we provide a compact generating set for $\KIC$ that will be used in our next results in this paper.
\end{remark}

\section{Gorenstein projective precover of complexes}
In \cite{EEI}, the authors proved that if a commutative noetherian ring admits a dualising complex, then every right bounded complex of modules has a Gorenstein projective precover. In this section we generalize this result to any complex, not necessarily right bounded, over any commutative noetherian ring of finite Krull dimension.

First, let us provide a characterization of totally acyclic complexes of projectives and injectives in $\K(\C(R))$ in terms of the associated row complexes.

\begin{proposition}\label{qwe}
Let $P \in \KPC$ and $I \in \KIC$. Then
\begin{itemize}
\item[(i)] $P \in \KTAPC$ if and only if $P^j \in \KTAPR$ for any $j \in \Z.$
\item[(ii)] $I \in \KTAIC$ if and only if $I^j \in \KTAIR$ for any $j \in \Z.$
\end{itemize}
\end{proposition}

\begin{proof}
We just prove part $(i)$. Part $(ii)$ follows similarly. For the `only if' part, it is sufficient to prove that for any projective $R$-module $Q$ and any $n \in\Z$, the sequence \[\Hom_R(P^j_{n-1},Q)\lrt\Hom_R(P^j_{n},Q)\lrt\Hom_R(P^j_{n+1},Q)\] is exact. This follows from the following diagram that exists thanks to the adjoin pair $(k^j,k^j_{\rho})$ and also the fact that the upper row is exact.
\[ \xymatrix@R-0.5pc{\Hom_{\C(R)}(P_{n-1},k_{\rho}^j(Q)) \ar[r] \ar[d]^{\simeq} & \Hom_{\C(R)}(P_n,k_{\rho}^j(Q)) \ar[r] \ar[d]^{\simeq} & \Hom_{\C(R)}(P_{n+1},k_{\rho}^j(Q)) \ar[d]^{\simeq} \\
\Hom_R(P^j_{n-1},Q) \ar[r] & \Hom_R(P^j_{n},Q) \ar[r] & \Hom_R(P^j_{n+1},Q).}\]

The `if' part can be settled in the same way. Just note that any projective complex is a product of projective complexes of the form $e^j_{\rho}(P)$, where $P$ is a projective module and $j \in \Z$. Here we used the fact that, for any $j \in \Z$, $e^j_{\la}(-)=e^{j-1}_{\rho}(-)$.
\end{proof}

Recall that if $\CT$ is a triangulated category with suspension functor $\Sigma$, an additive functor $\HT:\CT \rt \Ab$ is said to be homological provided it sends triangles in $\CT$ to long exact sequences in $\Ab$, where $\Ab$ denotes the category of Abelian groups \cite[Definition 1.1.7]{N01}. The kernel of $\HT$, denoted $\Ker \HT$, is then the full subcategory of $\CT$ consisting of all objects $X$ such that $\HT(\Sigma^nX)=0$, for all integers $n$.

\begin{theorem}\label{Tmar}
Let $\CT$ be a compactly generated triangulated category. Let $\HT:\CT \rt \Ab$ be a coproduct-preserving homological functor. Then the inclusion $\Ker(\HT)\rt\CT$ admits a right adjoint.
\end{theorem}

\begin{proof}
The theorem is proved by Margolis \cite[\S 7]{M}. Also one may find a proof in \cite[\S 6]{K10}.
\end{proof}

\begin{proposition}\label{cxz}
Let $R$ be a commutative ring of finite Krull dimension. Then the inclusion $\KTAPC \hookrightarrow \KPC$ admits a right adjoint.
\end{proposition}

\begin{proof}
First note that by Corollary \ref{RightCoherent}, $\KPR$ is compactly generated. Thus by Theorem \ref{main}, $\KPC$ is also compactly generated. We define a homological functor $\KPC \rt \Ab$ with kernel $\KTAPC$. Then the result follows from Theorem \ref{Tmar}. To this end, set $I=\oplus_{\mathfrak{p}} {\rm E}(R/{\mathfrak{p}})$, where $\mathfrak{p}$ runs over all prime ideals of $R$. Let $P \in \KPC$. We define $P\otimes I$ to be the induced complex obtaining from $P$ by tensoring any row (or equivalently any column) of $P$ with $I$. So for any $i \in \Z$, $(P\otimes I)^i={P^i}\otimes I$. For any $P \in \KPC$, define $\HT(P)$ by setting
\[\HT(P)=\HT_0(P)\oplus \HT_0(P\otimes I),\]
where $\HT_0$ is the $0$-th homology functor. Obviously this defines a homological functor $\HT:\KPC \lrt \Ab$. Note that an object $P$ of $\KPC$ belongs to the kernel of $\HT$ if and only if it is acyclic and for any $i$ and $j$, $\HT_j(P^i\otimes I)=0$. Now, one may apply Lemma 4.3 of \cite{MS} to get that $P\in\Ker\HT$ if and only if for any $i \in \Z$, $P^i$ is a totally acyclic complex of projective $R$-modules. Hence, by Proposition \ref{qwe}, $\Ker\HT=\KTAPC$. The proof is now complete.
\end{proof}

Once we know the existence of the adjoint for the inclusion $\KTAPC \rt \KPC$, we may apply an argument similar to \cite[\S 2]{J07} verbatim to prove the following theorem. We leave the details to the reader. One just should note that we have not used here the characterizations of Gorenstein projective complexes to show that their class is precovering, see Remark \ref{characterization}, below.

\begin{theorem}\label{Precovering}
Let $R$ be a ring of finite Krull dimension. Then the class of Gorenstein projective complexes is precovering in $\C(R)$.
\end{theorem}

Towards the end of this section, we plan to provide an outline of the proof of the fact that the argument we presented to show the existence of Gorenstein projective precoves can be dualized to show that the class of Gorenstein injective complexes are preenveloping. To this end we need the following general lemma of \cite{K05}.\\

\begin{lemma}\cite[Lemma 7.3]{K05}\label{krause}
Let $\CA$ be a locally noetherian Grothendieck category and suppose that $\D(\CA)$ is compactly generated. Then the inclusion $\iota:\KTAI \rt \KI$ has a left adjoint $\iota_{\la}: \KI \rt \KTAI$.
\end{lemma}

Let us go back to our situation and assume that $R$ is a noetherian ring. The noetherianness of $R$ implies that the category $\C(R)$ is a locally noetherian Grothendieck category. So to be able to apply the above lemma, one just needs to show that the derived category $\D(\C(R))$ is compactly generated. This can be achieved in the similar way as in the Theorem \ref{main} and using the following known fact: Let $X \in \D(\C(R))$. $X\cong 0$ if and only if for all $i \in \Z$, $X^i \simeq 0$ in $\D(R)$, or equivalently, $X^i$ is exact, for all $i \in \Z$.

So Lemma \ref{krause} can be applied to show that the inclusion $\iota:\KTAIC \rt \KIC$ has a left adjoint. An standard argument now can be applied to show that the class of Gorenstein injective complexes in $\C(R)$ is preenveloping.

\section{The homotopy categories $\K(\RGPrj)$ and $\K(\RGInj)$}
In this section we plan to study the homotopy category of Gorenstein projective and the homotopy category of Gorenstein injective $R$-modules. We divide the section into three subsections. In the first one we study the existence of adjoints and show that the inclusion $\iota: \K(\RGPrj) \lrt \KR$ (resp. $\iota: \K(\RGInj) \lrt \KR$) has a right (resp. left) adjoint. In the second one we show that when $R$ is a commutative noetherian ring admitting a dualising complex, there exists a triangle-equivalence $\KPC \simeq \KIC$. This will have a list of corollaries. In the last subsection, we show that under the same condition on the ring $R$, there exists a triangle equivalence $\K(\RGPrj)$ and $\K(\RGInj)$, which restricts to an equivalence $\KPR \simeq \KIR$.

\begin{remark}\label{characterization}
In \cite{R} and \cite{ER} it is shown that over an $n$-Gorenstein ring $R$, a complex $G$ of (left) $R$-modules is Gorenstein injective if and only if $G_i$ is Gorenstein injective $R$-module, for any $i \in \Z$. This result has been generalized recently \cite{LZ} to any left noetherian ring. The dual result for Gorenstein projective complexes is proved in \cite[\S 5]{EEI}. More precisely, they showed that if $R$ is a commutative noetherian ring of finite Krull dimension, then a complex $G$ of $R$-modules is Gorenstein projective if and only if $G_i$ is Gorenstein projective, for any $i \in \Z$. These results have been generalized to arbitrary rings in \cite{YL}. We use these characterizations throughout this section.
\end{remark}

\subsection{Existence of adjoint}
Let us begin this subsection, by recalling the following result from \cite{BEIJR}.

\begin{sproposition} \label{LOR}
Let $\CX$ be a class of objects of $\C(R)$ that is closed under extension and suspension. Let $M \in \C(R)$ and let $0 \rt Y \rt X \rt M \rt 0 $ be an exact sequence with $X \in \CX$ and $Y \in \CX^\perp$. Then, for any $X' \in \CX$, the induced homomorphism $\Hom_{\KR}(X',X) \rt \Hom_{\KR}(X',M)$ is a bijection.
\end{sproposition}

\begin{proof}
See Proposition 3.1 and Corollary 3.2 of \cite{BEIJR}.
\end{proof}

The following result should be compared with the Theorem 2.7 of \cite{G}, where it was shown that over an $n$-Gorenstein ring, the inclusion $\iota: \K(\RGPrj) \lrt \KR$ has a right adjoint. Recall that $R$ is called $n$-Gorenstein if it is two-sided noetherian and $\id_RR\leq n$ and $\id R_R\leq n$.

\begin{stheorem}\label{adjoint}
Let $R$ be a commutative noetherian ring of finite Krull dimension. The inclusion $\iota: \K(\RGPrj) \lrt \KR$ has a right adjoint.
\end{stheorem}

\begin{proof}
$(i)$ The functor $\iota:\K(\RGPrj) \rt \K(R)$ can be factored through $\K(\RGFlat)$ as in the following diagram
\[\xymatrix{ & \K(\RGFlat) \ar[dr]^{j} \\ \K(\RGPrj) \ar[ur]^l \ar[rr]^{\iota} & & \KR }\]
So it is enough to prove that each of the inclusions $l$ and $j$ have a right adjoint. Let us first show the existence of the right adjoint for $j$. By theorem 3.4 of \cite{YL} the pair $(\GFlat \ \C(R), {\GFlat \ \C(R)}^\perp)$ form a complete cotorsion theory in $\C(R)$. But by the characterization of Gorenstein flat complexes, we know that $(\GFlat \ \C(R))$ is nothing but the $\C(\RGFlat)$, see Lemmas 12 and 13 of \cite{EEI}. Now \ref{TEnochs} implies that $j$ has a right adjoint. To complete the proof, we should show that $l$ also has a right adjoint. By assumption, $R$ has finite Krull dimension, say $d$. Let $X \in \C(\RGFlat)$. We claim that the $d$-th syzygy of $X$ in $\C(R)$ is a Gorenstein projective complex. This follows from the fact that the Gorenstein projective dimension of any Gorenstein flat module is at most $d$ in view of the characterization of Gorenstein projective complexes \ref{characterization}. Therefore we conclude that the Gorenstein projective dimension of $X$ is at most $d$. A complex version of \cite[Theorem 2.10]{H} implies that there exists a short exact sequence
\[0 \rt D_X \rt G_X \rt X \rt 0\]
of complexes in which $G_X$ is a Gorenstein projective complex and the projective dimension of $D_X$ is at most $d-1$. Hence it follows that $D_X \in {(\GPrj\C(R))}^\perp$. We define $l_\rho:\K(\RGFlat) \rt \K(\RGPrj)$ by sending $X$ to $G_X$. Proposition \ref{LOR} now come to play to show that $l_\rho$ is in fact the right adjoint of $l$. Hence the proof is complete.
\end{proof}

\begin{sremark}
The proof of the above theorem provides another proof for the fact that over commutative noetherian rings of finite Krull dimension, the class of Gorenstein projective complexes is a precovering class in $\C(R)$, see Theorem \ref{Precovering}. This proof will use the characterizations of Gorenstein projective complexes recalled in \ref{characterization}. Let us provide an outline.

Let $X \in \C(R)$ be a complex of $R$-modules. It is known \cite[Theorem 2]{EEI} that the class of Gorenstein flat complexes is precovering. So $X$ admits a Gorenstein flat precover $F \st{\varphi}{\rt} X$. Since $F$ is a Gorenstein flat complex, the proof of the above theorem implies that $F$ has a Gorenstein projective precover $G \st{\psi}{\rt} F$. Now since $R$ is of finite Krull dimension, any Gorenstein projective complex is Gorenstein flat, and so we may deduce that $G \st{\varphi\psi}{\lrt} X$ is a Gorenstein projective precover of $X$.
\end{sremark}

\begin{sremark}
Another important fact is that the argument mentioned in the above remark also works for modules instead of complexes to provide another easy proof for the fact that over commutative noetherian rings of finite Krull dimension, the class of Gorenstein projective modules is precovering, see \cite[Corollary 2.13]{J07} and \cite[Theorem A.1]{MS}. The only facts that one should use is that over such rings the Gorenstein projective dimension of any Gorenstein flat module is bounded by the dimension of the ring and also over such rings any Gorenstein projective module is Gorenstein flat. These two are both known in the literature.
\end{sremark}

\subsection{Equivalence of $\KPC$ and $\KIC$}
Iyengar and Krause proved that if $R$ is a commutative noetherian ring with a dualising complex, there exists an equivalence of triangulated categories $\KPR \simeq \KIR,$ which is given by tensoring with the dualising complex \cite[Theorem 4.2]{IK}. Let us refer to this equivalence as IK-equivalence. In this subsection, we generalize this result and get a triangle-equivalence between triangulated categories $\KPC$ and $\KIC$, provided $R$ is a commutative noetherian ring admitting a dualising complex.

Let us begin by looking more carefully at the IK-equivalence. Let $R$ be a commutative noetherian ring admitting a dualising complex $D$. The quasi-inverses functors of the IK-equivalence are shown in the following diagram
\[ \xymatrix@C+3pc{ \KPR \ar@<-1ex>[r]_i & \KFR \ar@<-1ex>[r]_{D\otimes_R -} \ar@<-1ex>[l]_{q} & \KIR, \ar@<-1ex>[l]_{\Hom_R(D,-)} } \]
It is clear that the pair $(D\otimes_R - ,\Hom_R(D,-))$ is an adjoint pair. Since in this situation, $\KPR$ is compactly generated and $i$ preserves coproducts, Lemma \ref{lBrown} implies the $i$ admits a right adjoint $q$, see \cite[Remark 3.2]{IK}. So $(i,q)$ also form an adjoint pair. Therefore so is their composition
$(T,S):=((D\otimes_R -) \circ i, q\circ\Hom_R(D,-)).$ This, in particular implies that to show the equivalence between $\KPR$ and $\KIR$, it is enough to show that $T$ is an equivalence. This is what Iyengar and Krause did.

We first pass to the category of complexes. To this end, we define the functors
\[D\otimes - : \KFC \rt \KIC \ \ \ \ {\rm and} \ \ \ \ \Hom(D,-): \KIC \rt \KFC,\]
using the same notations as in the above but without prefix $R$, as follows. For any $F\in \KFC$, we define
$D\otimes F$ to be a complex in which for any $i \in \Z$, $(D\otimes F)^i=D\otimes_RF^i$. In fact, we apply the functor $D\otimes_R - $ on the rows of the complex $X$.
Similarly, for any object $I \in \KIC$, we let $\Hom(D,I)$ to be an object of $\KFC$ such that for any integer $i$, $(\Hom(D,I))^i=\Hom_R(D,I^i)$. It can be seen easily that, this will provide an adjoint pair $(D\otimes - ,\Hom(D,-))$ of functors
\[\xymatrix@C+3pc{\KFC \ar@<-1ex>[r]_{D\otimes -} & \KIC. \ar@<-1ex>[l]_{\Hom(D,-)}}\]

But here we have similar situations as in \cite[Remark 3.2]{IK}. That is $\KPC$ is compactly generated and the inclusion functor $i: \KPC \lrt \KFC$ preserves coproducts. So we may apply Lemma \ref{lBrown} to conclude that $i$ has a right adjoint $Q$. Hence we get the following diagram, in which the upper row is the right adjoint of the lower row.

\[ \xymatrix@C+3pc{ \KPC \ar@<-1ex>[r]_i & \KFC \ar@<-1ex>[r]_{D\otimes -} \ar@<-1ex>[l]_{Q} & \KIC, \ar@<-1ex>[l]_{\Hom(D,-)} } \]

Let us try to get a feeling about $Q$.

\begin{slemma}
Let $F \in \KFR$. Then for any $i \in \Z$, $Q(k^i_{\la}(F))\cong k^i_{\la}(q(F))$ in $\KPC$.
\end{slemma}

\begin{proof}
Fix $F \in \KFR$. Since $q$ is the right adjoint of $i$, there exists a triangle
\[\xymatrix@C-0.5pc{P \ar[r] & F \ar[r] & L \ar@{~>}[r] & }\]
in $\KFR$ with $P=q(F) \in \KPR$ and $L \in \KPR^\perp$. We may apply the functor $k^i_{\la}$ on this triangle, to get the exact triangle
\[\xymatrix@C-0.5pc{k^i_{\la}(q(F)) \ar[r] & k^i_{\la}(F) \ar[r] & k^i_{\la}(L) \ar@{~>}[r] & }\]
in $\KFC$. To complete the proof, it suffices to show that $k^i_{\la}(L) \in \KPC^\perp$, see \cite[4.1]{IK}. To show this, let $X \in \KPC$ be an arbitrary object. Here for simplicity we write Hom instead of $\Hom_{\KFC}$. We have
\[\Hom(X,k^i_{\la}(L))\cong \Hom(X,k^{i-1}_{\rho}(L)).\] The adjoint pair $(k^{i-1},k^{i-1}_{\rho})$ now implies the isomorphism
\[\Hom(X,k^{i-1}_{\rho}(L))\cong \Hom(X^{i-1},L).\] But the last group is zero, because $X^{i-1} \in \KPR$ and $L \in \KPR^\perp$.
\end{proof}

Using this lemma we can provide the following generalization of the IK-equivalence.

\begin{stheorem}\label{IK-equivCompl}
Let $R$ be a commutative noetherian ring admitting a dualising complex. There exists an equivalence of triangulated categories \[\KPC \st{\sim}{\rt} \KIC.\]
\end{stheorem}

\begin{proof}
For simplicity, set $U:=Q\circ\Hom(D,- )$ and $T:=(D\otimes - )\circ i$. We just show that there is a natural equivalence $\eta: \id_{\KPC} \lrt UT$. The other way around is similar. Since both categories are compactly generated, to prove the theorem, we may focus on their compact generating sets. By Theorem \ref{main}, the set
$\fS=\{ k^i_{\la}(P) : P \in S, i \in \Z\}$ is a compact generating set for $\KPC$, where $S$ is a compact generating set for $\KPR$. So it suffices only to show that $\eta_{k^i_{\la}}$, for any $i$ and any $P \in S$ is an isomorphism.
This follows from the following list of isomorphisms.
\[\begin{array}{lll}
UT(k^i_{\la}(P)) & \cong U(D\otimes k^i_{\la}(P))\\ & \cong U(k^i_{\la}(D\otimes_RP))\\ & \cong Q(k^i_{\la}(\Hom(D,D\otimes_RP)))\\ & \cong k^i_{\la}(q\circ \Hom(D,D\otimes_RP)).
\end{array}\]
It follows from \cite[Theorem 4.2]{IK} that the last term is isomorphic to $k^i_{\la}(P)$. This is what we want. Hence the proof is complete.
\end{proof}

We need the following technical lemma in the proof of our next result.

\begin{slemma}
Let $q$ and $Q$ be the above mentioned functors. Then for any object $F \in \KFC$, we have ${Q(F)}^i=q(F^i)$.
\end{slemma}

\begin{proof}
Let $F$ be an arbitrary object of $\KFC$. Since $Q$ is the right adjoint of the inclusion functor $i:\KPC \rt \KFC$, there exists a triangle
\[\xymatrix@C-0.5pc{Q(F) \ar[r] & F \ar[r] & L \ar@{~>}[r] & }\]
in $\KFC$, such that $Q(F)\in \KPC$ and $L \in {\KPC}^\perp$. Let $i \in \Z$ be an integer. By applying the $i$th evaluation functor on the above triangle, we get the exact triangle
\[\xymatrix@C-0.5pc{{Q(F)}^i \ar[r] & F^i \ar[r] & L^i \ar@{~>}[r] & }\]
in $\KFR$. So to complete the proof it suffices to show that $L^i \in {\KPR}^\perp$, see e.g. \cite[4.1]{IK}. To this end, let $P \in \KPR$ and consider the group $\Hom_{\KFR}(P,L^i)$. The adjoint pair $(k^i_{\la},k^i)$ implies the following isomorphism of abelian groups
\[\Hom_{\KFR}(P,L^i) \cong \Hom_{\KPC}(k^i_{\la}(P),L).\] But the latter group is zero because $L \in {\KPC}^\perp$.
The proof is hence complete.
\end{proof}

\begin{sproposition}\label{equivTotAcyc}
Let $R$ be a commutative noetherian ring admitting a dualising complex $D$. The functor $D\otimes - $ induces a triangle-equivalence between categories $\KTAPC$ and $\KTAIC$ such that the following diagram is commutative.
\[\xymatrix@R-1pc{\KPC \ar[r]^{D\otimes - } & \KIC \\ \KTAPC \ar[u] \ar[r]^{D\otimes - } &  \KTAIC \ar[u] & }\]
\end{sproposition}

\begin{proof}
Let $P \in \KTAPC$. By definition $D\otimes P$ is an object of $\KIC$ such that for any $i \in \Z$,
$(D\otimes P)^i=D\otimes P^i$. It follows from \cite[Proposition 5.9(i)]{IK} that $D\otimes P^i$ in fact belongs to $\KTAIR$. Now part $(ii)$ of Proposition \ref{qwe}, implies that $D\otimes P \in \KTAIC$. For the converse, assume that $I$ is an arbitrary object of $\KTAIC$. To complete the proof, it suffices to show that $Q(\Hom(D,I))\in \KTAPC$. It is clear that $\Hom(D,I)$ is a complex in $\KFC$. By the above lemma, $(Q(\Hom(D,I)))^i=q((\Hom(D,I))^i)$. But by definition, $(\Hom(D,I))^i=\Hom(D,I^i)$. So $(Q(\Hom(D,I)))^i=q((\Hom(D,I^i)))$. Since $I^i$ is a complex in $\KFR$, the IK-equivalence implies that $q((\Hom(D,I^i))) \in \KTAPR$. Now the result follows from part $(i)$ of Proposition \ref{qwe}.
\end{proof}

\begin{sremark}
Consider the above diagram once more.
\[\xymatrix@R-1pc{\KPC \ar[r]^{D\otimes - } & \KIC \\ \KTAPC \ar[u]^{\ell} \ar[r]^{D\otimes - } &  \KTAIC \ar[u]^{\iota} & }\]
Since $\ell$ has a right adjoint and the rows are equivalences, it follows that $\iota$ also has a right adjoint. This, in particular, implies that the class of Gorenstein injective complexes, over a commutative noetherian ring with a dualising complex, is precovering. One also may use the left adjoint of $\iota$ to deduce that $\ell$ admits a left adjoint and then deduce that the class of Gorenstein projective modules over these rings is preenveloping.
\end{sremark}

\begin{sremark}\label{KCPC}
We let $\KCPC$ denote the full subcategory of $\KTAPC$ consisting of objects $P$ such that all its rows are projective complexes, in notation, $P^i$ is a projective complex, for any $i \in \Z$. So all rows and all columns of $P$ are projective complexes. The subscript `c' comes from the word `contractible'. Similarly, we denote the full subcategory of $\KTAIC$ consisting of all objects $I$ such that for any integer $i$, $I^i$ is an injective complex, by $\KCIC$.
It is easy to see that the equivalence $\KTAPC \st{\sim}{\rt} \KTAIC$ restricts to an equivalence $\KCPC \st{\sim}{\rt} \KCIC$, that is, the following diagram is commutative.
\[\xymatrix{\KTAPC \ar[r]^{D\otimes - } & \KTAIC \\ \KCPC \ar[u]^{\nu} \ar[r]^{D\otimes - } & \KCIC \ar[u]^{\mu} }\]
where $\nu$ and $\mu$ are inclusions.
\end{sremark}

\begin{sremark}\label{stable}
Let $R$ be a ring. One may check easily that both categories $\GPrj \ \C(R)$ and $\GInj \ \C(R)$ are Frobenius categories with respect to the natural structures inducing from the short exact sequences in $\C(R)$. Thus their stable categories carries triangulated structures. Let us denote them by $\GPrj \ \underline{\C(R)}$ and $\GInj \ \overline{\C(R)}$. Proposition 7.2 of \cite{K05} and its dual, imply that there are triangulated equivalences as follows:
\[\KTAPC \st{\sim}{\rt} \GPrj \ \underline{\C(R)} \ \ \ {\rm and} \ \ \  \KTAIC \st{\sim}{\rt} \GInj \ \overline{\C(R)}.\]
\end{sremark}

\begin{sproposition}\label{ComGen}
Let $R$ be a commutative noetherian ring admitting a dualising complex. Then the triangulated categories $\KTAPC$ and $\KTAIC$ are compactly generated.
\end{sproposition}

\begin{proof}
By Proposition \ref{equivTotAcyc}, these two categories are equivalence. So it suffices for us to prove that $\KTAIC$ is compactly generated. This we do. We may apply Lemma \ref{krause}, to deduce that the inclusion $\iota: \KTAIC \rt \KIC$ has a left adjoint $\iota_{\la}$. Note that by our assumption, $\KIC$ is compactly generated. Let $S$ be a compact generating set for $\KIC$. We claim that the set $\{\iota_{\la}(I) : I \in S \}$ is a compact generating set for $\KTAIC$. Let $\{X_{\al}\}_{\al \in J}$ is a family of objects indexed by the set $J$. The adjoint pair $(\iota_{\la},\iota)$ implies the following isomorphism
\[\Hom_{\KTAIC}(\iota_{\la}(I),\coprod_{\al \in J}X_{\al}) \cong \Hom_{\KIC}(I,\coprod_{\al \in J}X_{\al}).\]
But $I \in S$ means that the second Hom is isomorphic to the $\coprod_{\al \in J}\Hom_{\KIC}(I,X_{\al}).$ Another use of the adjoint duality implies the isomorphism
\[\coprod_{\al \in J}\Hom_{\KIC}(I,X_{\al})\cong \coprod_{\al \in J}\Hom_{\KTAIC}(\iota_{\la}(I),X_{\al}).\]
This implies that $\iota_{\la}(I)$, for any $I \in S$, is a compact object. To show that they also form a generating set,
let $X \in \KTAIC$ be a non-zero object. So there is an element $I \in S$ such that $\Hom_{\KIC}(I,X)\neq 0$. The adjoint duality implies that $\Hom_{\KTAIC}(\iota_{\la}(I),X)\neq 0$. This completes the proof.
\end{proof}

The above proposition in view of Remark \ref{stable} implies the following result.

\begin{scorollary}\label{stableEquiv}
Let $R$ be a commutative noetherian ring admitting a dualising complex. Then the triangulated categories $\GPrj \ \underline{\C(R)}$ and $\GInj \ \overline{\C(R)}$ are compactly generated.
\end{scorollary}

\subsection{Equivalence of $\K(\RGPrj)$ and $\K(\RGInj)$}
Recently, Chen \cite[Theorem B]{C} provide an equivalence between the triangulated categories $\K(\RGPrj)$ and $\K(\RGInj)$, the homotopy category of Gorenstein projective and Gorenstein injective $R$-modules, assuming that $R$ is left-Gorenstein. He showed that in case $R$ is a commutative Gorenstein ring, up to a natural isomorphism, this equivalence extends IK's one. We recall that a ring $R$ is called left-Gorenstein \cite{B} if any (left) $R$-module is of finite projective dimension if and only if it is of finite injective dimension.
In this subsection we plan to show that there exists a triangle-equivalence between triangulated categories $\K(\RGPrj)$ and $\K(\RGInj)$ that restricts to an equivalence between $\KPR$ and $\KIR$. Actually we do not know if this is an extension of the IK-equivalence. Let us begin by the following lemma.

\begin{slemma}\label{equNotTri}
There exist functors of categories
\[\KTAPC {\lrt} \K(\RGPrj) \ \ \ \ {\rm and} \ \ \ \ \KTAIC {\lrt} \K(\RGInj).\]
\end{slemma}

\begin{proof}
Let $X \in \KTAPC$. Define a functor $\xi: \KTAPC \rt \K(\RGPrj)$ by $\xi(X)=Z_0(X)$, where $Z_0(X)$ denotes the kernel $\Ker(X_0 \rt X_{-1})$. To show that this is really a functor, it suffices to show that if a map $\varphi: X \lrt Y$ is zero in $\KTAPC$, then the induced map $\xi(\varphi): Z_0(X) \lrt Z_0(Y)$ is zero in $\K(\RGPrj)$. Since by Remark \ref{stable},
$\KTAPC \simeq \GPrj \ \underline{\C(R)}$, the vanishing of $\varphi$ in $\KTAPC$ is equivalent to say that $\xi(\varphi)$ factors through a projective complex $P$. By the construction of  projective complexes, we know that $P$ can be written as a direct sum $\coprod_{i \in \Z}e^i_{\la}(P_i)$,
\[\xymatrix{\cdots \ar[r] &  P_2\oplus P_1 \ar[r]^{\delta_1} & P_1\oplus P_0 \ar[r]^{\delta_0} & P_0\oplus P_{-1} \ar[r]^{\delta_{-1}} & \cdots,}\]
where for any integer $i$ in $\Z$, $P_i$ is a projective $R$-module and $\delta_i$ is defined as $\delta_i|_{P_{i+1}}=0$ and $\delta_i\mid_{P_i}=\id_{P_i}$.
Based on this, for any $i$, we introduce a morphism $s_i:{Z_0(X)}_i \lrt {Z_0(Y)}_{i+1}$ that construct the desired homotopy. Let us for simplicity denote the complex $Z_0(X)$ by $U$ and $Z_0(Y)$ by $V$. Consider the following diagram
\[\xymatrix{\cdots \ar[r] & U_{i+2} \ar[r]^{\pa_{i+2}} \ar@/_0.5pc/@{.>}[dd] & U_{i+1} \ar[r]^{\pa_{i+1}} \ar[dl]_{\beta_{i+1}} \ar@/_0.5pc/@{.>}[dd]^>>>{\xi(\varphi)_{i+1}} & U_{i} \ar[dl]_{\beta_{i}} \ar[r] \ar@/_0.5pc/@{.>}[dd] & \cdots  \\ \cdots \ar[r] & P_{i+1}\oplus P_i \ar[dr]^{\gamma_{i+1}} \ar[r]^{\delta_i} & P_i\oplus P_{i-1} \ar[dr]^{\gamma_{i}} \ar[r]^{\delta_{i-1}} & P_{i-1}\oplus P_{i-2} \ar[r]  & \cdots \\ \cdots \ar[r] & V_{i+2} \ar[r]^{\nu_{i+2}} & V_{i+1} \ar[r]^{\nu_{i+1}} & V_{i} \ar[r] & \cdots }\]
where for any $i$, $\xi(\varphi)_i=\gamma_i\beta_i$. Define $s_i:U_i \lrt V_{i+1}$ by $s_i=\gamma_{i+1}{(\delta_i\mid_{P_i})}^{-1}\beta_i$. It is an easy diagram chasing to show that $s$ is the desired homotopy.
\end{proof}

\begin{sdefinition}
Let $\varphi: X \rt Y$ be a morphism of bicomplexes. We say that $\varphi$ is vertically-null-homotopic (v-null-homotopic for short) if for any integers $i, j \in \Z$, there exists homomorphism $s_i^j:X_i^j \rt Y_i^{j+1}$ such that (with the notation as in \ref{notation})
\begin{itemize}
\item [$(1)$] $\varphi_{i}^{j}=s_{i}^{j-1}(v_x)_{i}^{j}+(v_y)_{i}^{j+1}s_{i}^{j},$ and
\item [$(2)$] $(h_y)_{i}^{j+1}s_{i}^{j}=s_{i-1}^{j}(h_x)_{i}^{j}.$
\end{itemize}
\end{sdefinition}

Note that the condition $(1)$ means that the morphism $\varphi$ restricted to any column is null-homotopic, while the second condition means that $s^j:X^j \rt Y^{j+1}$ is a morphism of complexes. Saying roughly, this is a vertical version of the definition of a null-homotopic map in the category of bicomplexes.

\begin{slemma}\label{lifting}
Let $f:G \rt G'$ be a null-homotopic morphism in $\K(\RGPrj)$. Let $X=T_G$ and $Y=T_{G'}$ denote respectively totally acyclic complexes in $\KTAPC$ with $Z_0X=G$ and $Z_0Y=G'$. Then $f$ can be lifted to a v-null-homotopic map $\varphi:X \rt Y$. Similar statement holds true whenever $f$ is a morphism in $\K(\RGInj)$.
\end{slemma}

\begin{proof}
Let $s$ be the homotopy morphism that forces $f$ to be null-homotopic. Since by Proposition \ref{qwe}, the rows of $T_G$ and $T_{G'}$ are totally acyclic, an standard argument shows that $s$ in each $j$th row induces a morphism $\bar{s}^j:X^j \lrt Y^{j+1}$. For any integers $i$ and $j$ we define the map $\varphi_i^j: X_i^j \lrt Y_i^j$ by setting $\varphi_i^j:=\bar{s}_{i}^{j-1}(v_x)_{i}^{j}+(v_y)_{i}^{j+1}\bar{s}_{i}^{j}.$ It is routine to check that $\varphi$ is a v-null-homotopic morphism of bicomplexes.
\end{proof}

\begin{sremark}\label{cone}
Let $f:X \rt Y$ be a morphism of complexes in $\C(R)$. By \cite{BEIJR} there exists a short exact sequence
\[0 \lrt Y \lrt \cone(f) \lrt \TS(X) \lrt 0\]
in $\C(R)$, in which $\cone(f)$ denotes the mapping cone of $f$ and $\TS(X)$ denotes the suspension of $X$, see \ref{suspension}.
It can be checked easily that a commutative diagram
\[\xymatrix{X \ar[r]^f \ar[d]_\rho & Y \ar[d]^\mu \\ X' \ar[r]^{f'} & Y'}\]
in $\C(R)$ induces the commutative diagram
\[\xymatrix{0 \ar[r] & Y \ar[r] \ar[d]^\mu & \cone(f) \ar[r] \ar@{.>}[d] & \TS(X) \ar@{.>}[d] \ar[r] & 0 \\ 0 \ar[r] & Y' \ar[r] & \cone(f) \ar[r] & \TS(X') \ar[r] & 0 }\]
with exact rows.
\end{sremark}

The following lemma is proved in \cite[Lemma 3.2]{GT}. See also \cite[Lemma 2.1]{BEIJR}.

\begin{slemma}\label{enochs}
The above short exact sequence splits if and only if $f\simeq 0$.
\end{slemma}
\vspace{0.3cm}

Now assume that $\varphi:X \rt Y$ is a morphism of bicomplexes. By the Remark \ref{cone}, we have the following commutative diagram in $\C(R)$ with the exact columns.
\[\xymatrix@C-0.5pc@R-0.8pc{& 0 \ar[d] & 0 \ar[d] & 0 \ar[d] \\ \cdots \ar[r] & Y_{1} \ar[r] \ar[d] & Y_{0} \ar[r] \ar[d] & Y_{-1} \ar[r] \ar[d] & \cdots \\ \cdots \ar[r] & \cone(\varphi_{1}) \ar[r] \ar[d] &\cone(\varphi_{0})  \ar[r] \ar[d] & \cone(\varphi_{-1}) \ar[r] \ar[d] & \cdots \\ \cdots \ar[r] & \TS(X_1) \ar[r] \ar[d] & \TS(X_{0}) \ar[r] \ar[d] & \TS(X_{-1}) \ar[r] \ar[d] & \cdots \\ & 0 & 0 & 0  }\]

We say that a morphism $\varphi:X \rt Y$ of bicomplexes is degree-wise null-homotopic (dw-null-homotopic, for short) if for each $i \in \Z$, the morphism $\varphi_i: X_i \rt Y_i$ is null-homotopic.

Therefore, if $\varphi:X \rt Y$ is a dw-null-homotopic map, by Lemma \ref{enochs}, all columns of the above diagram split and by an standard argument we get a triangle
\[\xymatrix@C-0.5pc{Y \ar[r] & \vcone(\varphi) \ar[r] & \VS(X) \ar@{~>}[r] & }\]
in $\KC$. Note that in this triangle $\vcone$ is in fact the vertical (not the usual) cone of $\varphi$. This explain the notation we used. Moreover, $\VS(X)$ is the bicomplex with $\VS(X)^j=X^{j-1}$, that is shifting the rows of $X$ one degree to the up, which is again different from the usual shifting in $\KC$.\\

In particular we have the following corollary.

\begin{scorollary}
Let $X \in \KPC$. Then any morphism $\varphi:X \rt Y$ induces a triangle as above.
\end{scorollary}

\begin{proof}
Since $X \in \KPC$ all columns of $X$ are projective complexes and hence are contractible. So for any $i \in \Z$, the morphism $\varphi_i: X_i \rt Y_i$ is null-homotopic. The result now follows from the above discussion.
\end{proof}

By definition, any v-null-homotopic map is dw-null-homotopic. We showed that any dw-null-homotopic morphism induces a triangle in $\KC$. In the following lemma we show that if $\varphi$ is v-null-homotopic, the induced triangle will be split.

\begin{slemma}
Let $\varphi:X \rt Y$ be a v-null-homotopic morphism of bicomplexes. Then the induced exact sequence
\[\xymatrix@C-0.5pc{0 \ar[r] & Y \ar[r] & \vcone(\varphi) \ar[r] & \VS(X) \ar[r] & 0}\] is split exact.
\end{slemma}

\begin{proof}
Consider the diagram
\[\xymatrix@C-0.5pc@R-0.8pc{& 0 \ar[d] & 0 \ar[d] & 0 \ar[d] \\ \cdots \ar[r] & Y_{1} \ar[r] \ar[d] & Y_{0} \ar[r] \ar[d] & Y_{-1} \ar[r] \ar[d] & \cdots \\ \cdots \ar[r] & \cone(\varphi_{1}) \ar[r] \ar[d]^{\pi_1} &\cone(\varphi_{0})  \ar[r] \ar[d]^{\pi_0} & \cone(\varphi_{-1}) \ar[r] \ar[d]^{\pi_{-1}} & \cdots \\ \cdots \ar[r] & \TS(X_1) \ar[r] \ar[d] & \TS(X_{0}) \ar[r] \ar[d] & \TS(X_{-1}) \ar[r] \ar[d] & \cdots \\ & 0 & 0 & 0  }\]
Since any column is split, for any $i$, we may find a map $u_i: \TS(X_i) \rt \cone(\varphi_i)$ with $\pi_iu_i=\id_{\TS(X_i)}$, see the proof of \cite[Lemma 2.1]{BEIJR}. Since $\varphi$ is v-null-homotopic, it is easy to check that the collection $\{u_i\}_{i \in \Z}$ is a morphism $\VS(X) \rt \vcone(\varphi)$. This implies the result.
\end{proof}

Our next result can be considered as a converse of the above lemma.

\begin{slemma}\label{dw-null-hom}
Let $\varphi:X \rt Y$ be a dw-null-homotopic morphism of bicomplexes with the property that the induced triangle
\[\xymatrix{Y \ar[r]^{\iota} & \vcone(\varphi) \ar[r]^{\pi} & \VS(X) \ar@{~>}[r] & }\]
splits. Then there exists a v-null-homotopic map $\psi:X \rt Y$ which is homotopic to $\varphi$ in $\KC$.
\end{slemma}

\begin{proof}
Our assumption implies there exists a morphism $u:\VS(X) \lrt \vcone(\varphi)$ such that $\pi u-\id_{\VS(X)}: \VS(X) \rt \VS(X)$ is null-homotopic. We define $\psi:X \rt Y$ by setting $\psi:=\varphi\circ\VS^{-1}(\pi u)$. Since $\psi-\varphi=\varphi\circ\VS^{-1}(\pi u-\id_{\VS(X)})$ and $\pi u-\id_{\VS(X)}$ is null homotopic, we may deduce that $\psi$ is homotopic to $\varphi$. To see that $\psi$ is a v-null-homotopic map we may define, $s^j: X^j \rt Y^{j+1}$, for any $j \in \Z$ by defining $s^j_i:X^j_i \rt Y^{j+1}_i$ to be $s^j_i=\nu^j_iu^j_i$, where $\nu$ is the canonical projection map $\vcone(\varphi) \rt Y$. Now it is just a simple diagram checking to show that $\psi$ is, in fact, a v-null-homotopic map.
\end{proof}

Now we are ready to present our last theorem in this paper.

\begin{stheorem}\label{main2}
Let $R$ be a commutative noetherian ring admitting a dualising complex $D$. Then there exists a triangle-equivalence between triangulated categories $\K(\RGPrj)$ and $\K(\RGInj)$, that restricts to an equivalence between $\KPR$ and $\KIR$.
\end{stheorem}

\begin{proof}
We plan to introduce a functor $\Psi$ that commutes the following diagram
\[\xymatrix{\K(\RGPrj) \ar@{.>}[rr]^{\Psi} & & \K(\RGInj) \\ \KPR \ar[u] \ar@{.>}[rr]^{\Psi|} & & \KIR \ar[u]}\]
Let $G \in \K(\RGPrj)$. By definition, there exists a totally acyclic complex $T_G$ in $\KTAPC$ such that $Z_0T_G=G$. We apply the equivalence $D\otimes - :\KTAPC \rt \KTAIC$ and define $\Psi(G)$ to be the complex $Z_0(D\otimes T_G)$. Clearly $\Psi(G) \in \K(\RGInj)$. Let $f:G \rt G'$ be a morphism in $\K(\RGPrj)$ which is null-homotopic. By Lemma \ref{lifting} $f$ can be lifted to a v-null-homotopic morphism $\varphi: T_G \rt T_{G'}$. Therefore the short exact sequence
\[0 \lrt T_{G'} \lrt \vcone(\varphi) \lrt \VS(X) \lrt 0\]
is split exact. This implies that the induced exact sequence
\[0 \lrt D\otimes T_{G'} \lrt \vcone(D\otimes \varphi) \lrt D\otimes \VS(X) \lrt 0\]
is also split exact. Hence by Lemma \ref{dw-null-hom}, there exists a v-null-homotopic map $\psi:X \rt Y$ which is homotopic to $D\otimes\varphi$ in $\KC$. This in view of Lemma \ref{equNotTri}, implies that $\Psi(f):\Psi(G) \lrt \Psi(G')$ is null-homotopic. Hence we have proved that $\Psi$ is well-defined. Following similar argument, will imply that $\Psi$ is faithful. Furthermore, it follows easily, by applying the equivalence $\KPC \simeq \KIC$ of Theorem \ref{IK-equivCompl}, that $\Psi$ is full and dense.

For the lower row of the diagram, just note that by the proof of Theorem 2.2 of \cite{YL}, if $G$ is a complex in $\KPR$ then there exists a totally acyclic complex $T_G \in \KTAPC$ with $G=Z_0T_G$ and with the extra property that it belongs to $\KCPC$. Now the proof follows from the Remark \ref{KCPC}.
\end{proof}

\section*{Acknowledgments}
The authors thank the Center of Excellence for Mathematics (University of Isfahan).

\end{document}